%% file: main.tex
\documentclass[11pt]{amsart}

\include{preamble}

\usepackage[utf8]{inputenc}

\title{Moduli of Representations of Skewed-Gentle Algebras}

\begin{document}
\author{Cody Gilbert}
\address{Saint Louis University, Department of Mathematics and Statistics, Saint Louis, USA}
\email[Cody Gilbert]{cody.gilbert@slu.edu}

\maketitle

\begin{abstract}
We prove irreducible components of moduli spaces of semistable representations of skewed-gentle algebras, and more generally, clannish algebras, are isomorphic to products of projective spaces. This is achieved by showing irreducible components of varieties of representations of clannish algebras can be viewed as irreducible components of skewed-gentle algebras, which we show are always normal. The main theorem generalizes an analogous result for moduli of representations of special biserial algebras proven by Carroll-Chindris-Kinser-Weyman.
\end{abstract}

\tableofcontents

\section{Introduction}
Throughout this paper, $k$ will be an algebraically closed field of characteristic zero, all quivers will be finite and connected and all algebras will be assumed to be associative and finite-dimensional over $k$.

In this paper, we view representations of algebras through the lens of Geometric Invariant Theory (GIT) in order to construct moduli spaces of representations, which were introduced by King in \cite{KING}. These moduli spaces of representations are projective varieties whose points are in bijection with isomorphism classes of polystable representations. Work dedicated to understanding moduli of representations of algebras includes \cite{BobSkow, DoLe, GeissSchroer,Reineke2, Reineke3, Domokos, Reineke1, tametilted, Weist, FlorLaw, quasitilted, ChiCar, gentle, ReiSch, ModDec, Hoskins1, Hoskins2, ModSpBi, HosSchaf, AmruDubey, Franzen, ChiKli, FrReiSaba}. 

In general, moduli of representations can be arbitrarily complicated, in the sense that any projective variety is expressible as a moduli space of representations \cite{Hille, Huisgen}. With this in mind, we impose conditions on the algebras of consideration in order to obtain more precise descriptions of their moduli spaces. For example, from the above observation, there is no hope in classifying moduli spaces of representations of wild algebras; however, when restricting our focus to tame algebras there is still hope that the following classification will hold:

\begin{conjecture}\cite[\S 7]{gentle}\label{conj:conjecture}
Let $(Q, I)$ be a bound quiver, and $A=kQ/I$ its bound quiver algebra. If $A$ is of tame representation type, then for any irreducible component $Z\subset \text{rep}_Q(I, \textbf{d})$ and any weight $\theta$ such that $Z^{ss}_{\theta}\neq \emptyset$, $\mathcal{M}(Z)^{ss}_{\theta}$ is a product of projective spaces.
\end{conjecture}
Due to the lack of classification of tame algebras, this conjecture is unlikely to be proven
in the near future, so it makes sense to restict our focus to specific types of tame
algebras. In fact, for certain classes of tame algebras, the above decomposition of moduli
spaces as a product of projective spaces is known to hold. For instance, the
decomposition holds for hereditary algebras \cite{chindrisHereditary}, tame tilted algebras \cite[Proposition 4.4]{tametilted}, tame
quasi-tilted algebras \cite[Proposition 6.2]{quasitilted}, gentle algebras
\cite[Theorem 1]{gentle} and special biserial algebras \cite[Theorem 1.1]{ModSpBi}.

While there are numerous generalizations of gentle algebras and special biserial
algebras, see \cite{Benson, GreenSchroll2, GreenSchroll1}, in this paper, we will
focus specifically on skewed-gentle algebras \cite{GePe} and clannish algebras
\cite{CB2, Geiss}. Clannish algebras were introduced as a special class of clans, a
class of tame matrix problems, in \cite{CB2}. Like special biserial algebras, the
indecomposable representations of a clannish algebra are well known \cite{CB2} and
correspond to words that are either (symmetric) strings or (symmetric) bands. Clannish
algebras generalize special biserial algebras in the sense that all but finitely many
indecomposable representations of clannish algebras are pushed forward from $\mathbb{A}, \widetilde{\mathbb{A}}, \mathbb{D}$ and $\widetilde{\mathbb{D}}$ quivers, as opposed to just $\mathbb{A}$ and $\widetilde{\mathbb{A}}$ quivers in the special biserial case.
One can
think of skewed-gentle algebras and clannish algebras as being gentle algebras and
special biserial algebras, respectively, with the additional ability to now glue
idempotent loops to vertices. Work involving skewed-gentle algebras
includes \cite{ChenLu2, ChenLu, AB19, HeZhouZhu, LabFragSchr}, while applications for
clans and clannish algebras have surfaced in recent years in cluster theory, see
\cite{QiuZhou, AmiotPlam}.

In this paper, we generalize the classification of moduli spaces of representations of special biserial algebras found in \cite{ModSpBi} to the more general class of clannish algebras. Namely, the main theorem confirms that Conjecture \ref{conj:conjecture} holds for clannish algebras.

\begin{theorem}\label{thm:projmod}
Let $\Lambda=kQ/I$ be a clannish algebra (for example, a skewed-gentle algebra). Then any irreducible component of a moduli space $\mathcal{M}(A, \textbf{d})^{ss}_{\theta}$ is isomorphic to a product of projective spaces.
\end{theorem}

The proof of this theorem relies on showing certain irreducible components are normal varieties.
By interchanging gentle and special biserial algebras with skewed-gentle and clannish algebras, respectively,
one can extend the techniques utilized in \cite{ModSpBi} to prove Theorem \ref{thm:projmod}.

Explicitly, we show that any clannish algebra can be viewed as a quotient of a skewed-gentle algebra and irreducible representation varieties of such algebras are normal in Lemma \ref{lem:clantoskewgentle} and Lemma \ref{lem:skewgentlenormal}.
Via a dimension argument, we conclude irreducible components of
interest within the representation variety for the clannish algebra remain
irreducible components in the larger representation variety for the
skewed-gentle algebra and are thus normal by the previous observation.
An application of the general moduli decomposition theorem, proven in \cite{ModDec}, finishes the argument.

\vspace{5mm}
\noindent\textbf{Acknowledgements.} The author would like to thank Ryan Kinser for the invaluable guidance throughout the years, as well as Calin Chindris for the helpful comments and conversations.

\section{Moduli of Representations of Algebras}
\subsection{Representation Varieties}
A classical
result from Gabriel gives that any finite dimensional unital, associative $k$-algebra $A$ is Morita equivalent to a quotient $kQ/I$,
where $Q$ is a finite quiver uniquely determined by $A$ and $I$ is an admissible ideal of the path algebra $kQ$ generated by a collection $\mathcal{R}$ of linear combinations of paths of length at least two.
With this in mind, and by slight abuse of terminology, we refer to a module over $kQ$ annihilated by $I$ as a module over $A$. 
The explicit construction of $Q$ and $I$, given an algebra $A$, can be found in \cite[\S 2.3]{ASS06}.

For a quiver $Q$, we will let $Q_0$ denote the set of vertices and $Q_1$ be the set of arrows.
Further, $ta$ and $ha$ will denote the tail and head of an arrow $ta \xrightarrow{a} ha$, respectively.
We will assume throughout that $Q$ has finitely many arrows and vertices. 
For a fixed dimension vector $\textbf{d}\in\mathbb{N}^{Q_0}$, we define the affine representation variety $\text{rep}_Q(I,\textbf{d})$,
which parametrizes the $\textbf{d}$-dimensional representations of $A\cong kQ/I$ along with a fixed basis. 
Explicitly, 

\begin{equation*}\label{eqn:RepVar}
\text{rep}_Q(I,\textbf{d}):=\{M\in \prod_{a\in Q_1}\text{Mat}_{\textbf{d}(ha)\times \textbf{d}(ta)}(k)\,|\, M(r)=0, \text{ for all } r\in I\}.
\end{equation*}
In general, $\text{rep}_Q(I,\textbf{d})$ does not need to be irreducible. 

With $\text{GL}(\textbf{d})=\prod_{x\in Q_0} \text{GL}(\textbf{d}(x), k)$, we have an action on $\text{rep}_Q(I,\textbf{d})$ defined by 
\begin{equation}
    (\phi\cdot M)(a):=\phi(ha)M(a)\phi^{-1}(ta), \text{ where } a\in Q_1 \text{ and } \phi\in \text{GL}(\textbf{d}).
\end{equation}
One finds the orbits are in bijection with the isomorphism classes of $\textbf{d}$-dimensional representations over $A$. Additionally, the orbit-stabilizer theorem gives the equation
\begin{equation}\label{eqn:orbitstab}
    \dim \text{GL}(\textbf{d})\cdot M=\dim \text{GL}(\textbf{d})-\dim \text{Stab}_{\text{GL}(\textbf{d})}(M)=\dim\text{GL}(\textbf{d})-\dim\text{End}(M).
\end{equation}

An irreducible component $Z$ of $\text{rep}_Q(\langle \mathcal{R}\rangle,\textbf{d})$ is said to be indecomposable if its general points are indecomposable representations over $A$.
Similarly, $Z$ is a Schur component if the general points of $Z$ are Schur representations.
By definition, Schur components are indecomposable. 

For dimension vectors $\textbf{d}_i\in\mathbb{N}^{Q_0}$, and $\text{GL}(\textbf{d}_i)$-invariant constructible subsets $Z_i\subset \text{rep}_Q(I, \textbf{d}_i)$, $1\leq i\leq l$, we write 
\begin{equation*}
    Z_1 \oplus \ldots \oplus Z_l = \{M\in\text{rep}_Q(I, \sum_{i=1}^l\textbf{d}_i)\, |\, M\cong\bigoplus_{i=1}^l M_i \text{ with } M_i\in Z_i, \text{ for all } 1\leq i\leq l\} .
\end{equation*}
A result from de la Pe\~{n}a in \cite{dlP91} and Crawley-Boevey and Schro{\"e}r in \cite[Theorem 1.1]{CBS02} allows one to express any irreducible component $Z\subseteq\text{rep}_Q(I, \textbf{d})$ as a direct sum

\begin{equation}\label{eqn:irreducibledecomp}
    Z = \overline{Z_1\oplus\ldots\oplus Z_l}
\end{equation}
where the $Z_i\subseteq\text{rep}_Q(I, \textbf{d}_i)$ are indecomposable irreducible components with $\sum_{i=1}^l \textbf{d}_i=\textbf{d}$. Moreover, the $Z_1,\ldots, Z_l$ are uniquely determined by this decomposition, up to reordering.

When determining geometric properties of a representation variety over a finite dimensional algebra
$A=kQ/I$, it can be helpful to decompose the algebra into its $\rho$-block components
\cite{GLS}. Specifically, if $I$ is generated by a set $\{\rho_1,\ldots, \rho_m\}$ of
relations with $$\rho_k=\sum_{i=1}^s \lambda_i p_i,$$ and $Q(\rho_k)$ is the smallest
subquiver of $Q$ containing the paths $p_i$, then for arrows $a,b\in Q_1$, we write
$a\sim b$ if there is a $k\in \{1,\ldots m\}$ with $a,b\in Q(\rho_k)$.   

Each equivalence class with respect to $\sim$ in $Q$ determines a subquiver of $Q$ as
well as a subalgebra of $A$. These subalgebras are called the $\rho$-blocks of $A$. 

Let $A_1=kQ_1/I_1,\ldots, A_t=kQ_t/I_t$ be the $\rho$-blocks of $A$, $\textbf{d}\in \mathbb{N}^{|Q_0|}$ be a dimension vector, and for $1\leq j\leq t$, let $\pi_j(\textbf{d})$ be the corresponding dimension vector for $A_j$. Then each $M\in\text{rep}_Q(I,\textbf{d})$ induces an element $\pi_j(M)\in \text{rep}_{Q_j}(I_j, \textbf{d}_j)$ via restriction and we have the following isomorphism of affine varieties
\begin{align}\label{eqn: rho-block}
\text{rep}_Q(I, \textbf{d})&\rightarrow \text{rep}_{Q_1}(I_1, \pi_1(\textbf{d}))\times\ldots\times \text{rep}_{Q_t}(I_t, \pi_t(\textbf{d}))\\
M &\mapsto (\pi_1(M),\ldots, \pi_t(M)).\nonumber
\end{align}
The above isomorphism allows us to work with the simpler varieties
$\text{rep}_{Q_j}(I_j, \pi_j(\textbf{d}))$ 
in order to prove geometric properties of $\text{rep}_Q(I, \textbf{d})$.
For example, later in this paper we will prove irreducible components of varieties of skewed-gentle algebras are normal by showing the irreducible components of varieties over each $\rho$-block are normal.
\subsection{Semi-invariants}

The moduli spaces of interest will have coordinate rings given by semi-invariant rings, so we review spaces of semi-invariants briefly. 

The one-dimensional rational characters of $\text{GL}(n)$ are precisely the integer powers of determinants, so the characters of $\text{GL}(\textbf{d})$ are all of the form 

\begin{align*}
    \chi_{\theta}: \text{GL}(\textbf{d})& \rightarrow k^*\\
    (g_x)_{x\in Q_0}&\mapsto \prod_{x\in Q_0} \det (g_x)^{\theta(x)}
\end{align*}
where $\theta\in\mathbb{Z}^{Q_0}$ is an integral weight of $Q$. In this way, we get a natural isomorphism of $\mathbb{Z}^{Q_0}$ onto the group of rational characters of $\text{GL}(\textbf{d})$.   

For each rational character $\chi_{\theta}: \text{GL}(\textbf{d})\rightarrow k^{*}$ determined by a weight $\theta$, the eigenspace 
\begin{equation*}\label{eqn:SemiSpace}
\text{SI}(A,\textbf{d})_{\theta}=\{f \in k[\text{rep}_Q(I,\textbf{d})]\, | \, g\cdot f=\chi_{\theta}(g)f \text{ for all } g\in\text{GL}(\textbf{d})\}
\end{equation*}
is called the space of semi-invariants on $\text{rep}_Q(I,\textbf{d})$ of weight $\theta$.
If $\text{SI}(A,\textbf{d})_{\theta}\neq\{0\}$, then every 
$f\in \text{SI}(A,\textbf{d})_{\theta}\backslash\{0\}$ is a simultaneous eigenvector with eigenvalue $\chi_{\theta}(g)$ for all $g\in \text{GL}(\textbf{d})$. 
From these spaces we obtain the graded ring 
\begin{equation*}\label{eqn:SemiRing}
\bigoplus_{n\geq 0} \text{SI}(A,\textbf{d})_{n\theta}=k[\text{rep}_Q(I,\textbf{d})]^{\ker \chi_\theta}.
\end{equation*}
For $Z$ a $\text{GL}(\textbf{d})$-invariant, irreducible, closed subvariety of $\text{rep}_Q(I,\textbf{d})$, we can similarly define vector spaces $\text{SI}(Z)_{\theta}$ and hence obtain analogous graded rings for $Z$ as well. 

\subsection{Moduli Spaces of Representations}

The final ingredient required to define moduli spaces of representations involves introducing a notion of stability for affine representation varieties.
For this, we follow King's notion of stability found in \cite{KING}. 
With $\theta\in\mathbb{Z}^{Q_0}$ a weight and a dimension vector $\textbf{d}$, we define 
\begin{equation*}\label{eqn:DotProd}
\theta(\textbf{d})=\sum_{x\in Q_0}\theta(x)\textbf{d}(x).
\end{equation*}
\begin{definition}\label{def:Stability} {\ }

\begin{itemize}
\item A point $M\in\text{rep}_Q(I,\textbf{d})$ is $\theta$-semistable if $\theta(\text{\textbf{dim} } M)=0$ and $\theta(\text{\textbf{dim} } M')\leq 0$ for all $M'\leq M$.
\item A $\theta$-semistable representation $M$ is $\theta$-stable if $\theta(\text{\textbf{dim} } M')<0$ for all proper subrepresentations $M'$ of $M$.
\item A representation $M$ is $\theta$-polystable if it is a direct sum of $\theta$-stable representations.
\end{itemize}
\end{definition}

With the above definition in mind, we have two open (possibly empty) subsets
\begin{align*}
\text{rep}_Q(I, \textbf{d})^{ss}_{\theta}& =\{M\in \text{rep}_Q(I, \textbf{d})\, | \, M \text{ is } \theta\text{-semistable}\}\\
\text{rep}_Q(I,\textbf{d})^{s}_{\theta}& =\{M\in \text{rep}_Q(I, \textbf{d})\, | \, M \text{ is } \theta\text{-stable}\}
\end{align*}
of $\theta$(-semi)-stable representations of $A$ with dimension vector $\textbf{d}$.
The category $\text{rep}_Q(I)^{ss}_{\theta}$ of $\theta$-semistable representations of $A$ is abelian with simple objects precisely the $\theta$-stable representations.
In particular, Hom-spaces between $\theta$-stable representations either have dimension zero or one.
Further, two $\theta$-stable representations are $S$-equivalent if they have the same $\theta$-stable composition factors (counted with multiplicity).

In \cite{KING}, the projective spectrum is applied to the semi-invariant rings above to obtain the projective variety 
\begin{equation}\label{eqn:ModSpace}
\mathcal{M}(A, \textbf{d})^{ss}_{\theta}:= \text{Proj}(\bigoplus_{n\geq 0} \text{SI}(A,\textbf{d})_{n\theta}).
\end{equation}
The projective variety 
$\mathcal{M}(A, \textbf{d})^{ss}_{\theta}$
is a GIT quotient of 
$\text{rep}_Q(I, \textbf{d})^{ss}_{\theta}$
by the action of 
$\text{PGL}(\textbf{d})\leq \text{GL}(\textbf{d})$.
For $Z$ a $\text{GL}(\textbf{d})$-invariant, irreducible, closed subvariety of
$\text{rep}_Q(I, \textbf{d})$, we can define the locus of $\theta$(-semi)-stable representations 
$Z^{(s)s}_{\theta}$, as well as $\mathcal{M}(Z)^{(s)s}_{\theta}$, in an analogous manner.

By construction, the points of 
$\mathcal{M}(Z)^{ss}_{\theta}$ are in bijection with the closed, $\theta$-semistable $\text{GL}(\textbf{d})$-orbits in $Z$,
which are equivalent to the isomorphism classes of $\theta$-polystable representations in $Z^{ss}_{\theta}$.
In this sense, $\mathcal{M}(Z)^{ss}_{\theta}$ is a coarse moduli space for $\theta$-semistable representations of dimension vector $\textbf{d}$ up to $S$-equivalence.

In order to better understand $\mathcal{M}(Z)^{ss}_{\theta}$,
we would like to be able to relate the moduli space of $Z$ with the moduli spaces of its $\theta$-stable composition factors.
With this in mind, we begin with the following definition from \cite{ModDec}.

\begin{definition}\label{def:StableDecomp}
Let $Z$ be a $\text{GL}(\textbf{d})$-invariant, irreducible, closed subvariety of $\text{rep}_Q(I, \textbf{d})$, and assume $Z$ is $\theta$-semistable.
Consider a collection, $(Z_i\subseteq \text{rep}_Q(I,\textbf{d}_i))_i$ of $\theta$-stable irreducible components such that $Z_i\neq Z_j$ for $i\neq j$, along with a collection of multiplicities $(m_i\in \mathbb{Z}_{>0})_i$.
We say that $(Z_i, m_i)_i$ is a $\theta$-stable decomposition of $Z$ if, for a general representation $M\in Z^{ss}_{\theta}$, its corresponding $\theta$-polystable representation $\widetilde{M}$ is in $Z_1^{m_1}\oplus\ldots\oplus Z_l^{\oplus m_l}$, and write 
\begin{equation}\label{eqn:StableDecomp}
Z = m_1Z_1\dot{+}\ldots\dot{+}m_lZ_l.
\end{equation}
\end{definition} 
\noindent Any $\text{GL}(\textbf{d})$-invariant, irreducible, closed subvariety $Z$ of
$\text{rep}_Q(I, \textbf{d})$ with $Z^{ss}_{\theta}\neq \emptyset$ has a $\theta$-stable decomposition
by \cite[Proposition 3]{ModDec}.

The next theorem tells us that when analyzing moduli spaces of a given irreducible component $Z$,
we may assume that a general point of $Z$ is a direct sum of its $\theta$-stable composition factors,
while the multiplicities in the $\theta$-stable decomposition contribute symmetric powers to the overall moduli space.
As a reminder, the $m^{\text{th}}$ symmetric power $S^m(X)$ of a variety $X$ is the quotient of $\prod_{i=1}^m X$ by the action of the permutation group $S_m$.

\begin{theorem}\cite[Theorem 1]{ModDec}\label{Thm:DecModSp}
Let $A=kQ/I$ be a finite-dimensional algebra and let $Z\subseteq \text{rep}_Q(I, \textbf{d})^{ss}_{\theta}$
be a $\text{GL}(\textbf{d})$-invariant, irreducible, closed subvariety.
Let $Z=m_1Z_1\dot{+}\cdots\dot{+}m_rZ_r$ be a $\theta$-stable decomposition of $Z$ where
$Z_i\subseteq \text{rep}_Q(I, \textbf{d}_i), 1\leq i\leq r$, are pairwise distinct $\theta$-stable irreducible components,
and define $\widetilde{Z}=\overline{Z_1^{\oplus m_1}\oplus\cdots \oplus Z_r^{\oplus m_r}}$. 
\begin{enumerate}
    \item[(a)] If $\mathcal{M}(Z)^{ss}_{\theta}$ is an irreducible component of $\mathcal{M}(A, \textbf{d})^{ss}_{\theta}$, then 
    $$\mathcal{M}(\widetilde{Z})^{ss}_{\theta}=\mathcal{M}(Z)^{ss}_{\theta}.$$
    \item[(b)] If $Z_1$ is an orbit closure, then $$\mathcal{M}(\overline{Z_1^{\oplus m_1}\oplus\cdots\oplus Z_r^{\oplus m_r}})^{ss}_{\theta}\simeq\mathcal{M}(\overline{Z_2^{\oplus m_2}\oplus\cdots\oplus Z_r^{\oplus m_r}})^{ss}_{\theta}.$$
    \item[(c)] Assume now that none of the $Z_i$ are orbit closures. Then there is a natural morphism 
    $$\psi: S^{m_1}(\mathcal{M}(Z_1)^{ss}_{\theta})\times \cdots\times S^{m_r}(\mathcal{M}(Z_r)^{ss}_{\theta})\rightarrow \mathcal{M}(\widetilde{Z})^{ss}_{\theta}$$ which is finite and birational. In particular, if $\mathcal{M}(\widetilde{Z})^{ss}_{\theta}$ is normal then $\psi$ is an isomorphism.
\end{enumerate}
\end{theorem}

The irreducible components of the moduli space 
$\mathcal{M}(A, \textbf{d})^{ss}_{\theta}$ are all of the form $\mathcal{M}(Z)^{ss}_{\theta}$
with $Z$ a $\theta$-semistable irreducible component of $\text{rep}_Q(I, \textbf{d})$,
so the theorem takes all irreducible components of $\mathcal{M}(A, \textbf{d})^{ss}_{\theta}$ into account.   

\section{Background on Tame Algebras}
\subsection{Moduli spaces of tame algebras}

Throughout this section, $A=kQ/I$ will be a finite dimensional tame algebra, $\textbf{d}\in\mathbb{N}^{Q_0}$ a dimension vector and $\text{rep}_Q(I, \textbf{d})$ the corresponding representation variety.

As mentioned in the introduction, Conjecture \ref{conj:conjecture} is a problem of
interest that is known to hold for certain classes of tame algebras. While the goal of this paper is to prove the conjecture holds for clannish algebras, before focusing specifically on clannish algebras, we use this section to illustrate properties of irreducible components of representation varieties of tame algebras that are helpful in understanding moduli spaces of tame algebras. 

Despite the class of tame algebras being quite nebulous, the fact that one can
parametrize all but finitely many isoclasses of indecomposable representations of $Q$
of dimension vector $\textbf{d}$ by a finite number of one-parameter families allows
for the following
description of indecomposable irreducible components $Z\subseteq \text{rep}_Q(I, \textbf{d})$.
\begin{theorem}\cite[Lemma 2.3]{gentle}\cite[Theorem 3.1]{GLS}\label{thm:tameirreducible}
Let $A=kQ/I$ be a finite-dimensional tame algebra, and let $Z$ be an indecomposable irreducible component of $\text{rep}_Q(I, \textbf{d})$. Then 
\begin{equation}
    c_A(Z):=\min\{\dim(Z)-\dim\mathcal{O}_M\,|\, M\in Z\}\in\{0, 1\}.
\end{equation}
Furthermore,
\begin{itemize}
\item $c_A(Z)=0$ if and only if $Z$ contains an indecomposable representation $M$ with $$Z=\overline{\mathcal{O}_M},$$
\item $c_A(Z)=1$ if and only if $Z$ contains a rational curve $C$ such that the points of $C$ are pairwise non-isomorphic indecomposable representations with $$Z=\overline{\bigcup_{M\in C}\mathcal{O}_M}.$$
\end{itemize}
\end{theorem}

In conjunction with Theorem \ref{thm:tameirreducible}, Equation (\ref{eqn:orbitstab}) gives two possible equations for the dimension of an indecomposable irreducible component $Z\subset \text{rep}_Q(I, \textbf{d})$:
\begin{align}
    \text{When } Z=\overline{\mathcal{O}_M}:&\, \dim Z = \dim\text{GL}(\textbf{d})-\dim\text{End}(M)<\dim\text{GL}(\textbf{d})\label{eqn:dimcase1}\\
    \text{When } Z=\displaystyle{\overline{\bigcup_{M\in C}\mathcal{O}_M}}:&\,\dim Z=\dim\text{GL}(\textbf{d})-\dim \text{End}(M)+1\leq \dim\text{GL}(\textbf{d})\label{eqn:dimcase2}
\end{align}

The next lemma generalizes the bound on dimension established in the above equations for indecomposable irreducible components to arbitrary irreducible components of representation varieties for tame algebras.

\begin{lemma}\label{lem:dim}
If $Z\subseteq\text{rep}_Q(I, \textbf{d})$ is an irreducible component, then $\dim Z\leq \dim \text{GL}(\textbf{d})$.
\end{lemma}

\begin{proof}
In the case where $Z$ is an indecomposable irreducible component, we appeal to Theorem \ref{thm:tameirreducible} and Equations (\ref{eqn:dimcase1}) and (\ref{eqn:dimcase2}).

Moving to the general case, we can write 
\begin{equation*}
    Z = \overline{Z_1\oplus\ldots\oplus Z_l}
\end{equation*}
where the $Z_i\subseteq\text{rep}_Q(I, \textbf{d}_i)$ are indecomposable irreducible components with $\sum_{i=1}^l \textbf{d}_i=\textbf{d}$ by (\ref{eqn:irreducibledecomp}). Of the $l$-summands above, suppose $k$ of these summands are not expressible as orbit closures. As such, $Z$ has a dense $k$-parameter family, so for a general element $M\in Z$, Equations (\ref{eqn:dimcase1}) and (\ref{eqn:dimcase2}) yields
\begin{align*}\label{eq:dimension}
    \dim Z=\dim\text{GL}(\textbf{d})\cdot M+k&= \dim\text{GL}(\textbf{d})-\dim\text{End}(M)+k\\
    &\leq\dim\text{GL}(\textbf{d})+k-l\\
    &\leq\dim\text{GL}(\textbf{d}). \qedhere
\end{align*}
\end{proof}
The above bound on dimension and the next lemma will be key in allowing us to view certain
irreducible components of clannish algebras as irreducible components of
skewed-gentle algebras. The following generalizes parts of \cite[Proposition 4.3]{ModSpBi}.

\begin{lemma}\label{lem:tameirred}
Let $A=kQ/I$ and $B=kQ/I'$ be finite dimensional tame algebras with
$I'\subseteq I$. Let $Z_i\subseteq\text{rep}_Q(I, \textbf{d}_i),
1\leq i\leq m$, be irreducible components 
satisfying:
\begin{itemize}
    \item each $Z_i$ is a Schur component;
    \item each $c_A(Z_i) = 1$;
    \item $\Hom_A(M_i, M_j) = 0$ for $i \neq j$ and general points $M_i \in Z_i$, $M_j \in Z_j$.
\end{itemize}
Letting $\textbf{d}=\sum_{i=1}^m\textbf{d}_i$, then $Z=\overline{Z_1\oplus\ldots\oplus Z_m}$ is an irreducible
component of $\text{rep}_{Q}(I', \textbf{d})$ with regard to the
closed embedding $\text{rep}_Q(I, \textbf{d})\subset
\text{rep}_{Q}(I', \textbf{d})$.
\end{lemma}
\begin{proof}
By hypothesis, we write 
\begin{equation}
    Z=\overline{Z_1\oplus\ldots\oplus Z_m}\label{eqn:irrdirsum}
\end{equation}
where each $Z_i\subseteq \text{rep}_Q(I, \textbf{d}_i)$ is an irreducible Schur band component with $$\hom(Z_i, Z_j)=\min\{\dim\text{Hom}_{A}(M,N)\,|\, M\in Z_i, N\in Z_j\}=0$$ for all $1\leq i\neq j\leq m$.

As the $Z_i$ are Schur and are not orbit closures, by Lemma \ref{lem:dim}, it follows
$\dim Z_i=\dim \text{GL}(\textbf{d}_i)$ for each $i$. Viewing $Z_i$ inside $\text{rep}_{Q}(I',
\textbf{d}_i)$ via the closed embedding $\text{rep}_Q(I,\textbf{d}_i)\subset \text{rep}_{Q}(I',
\textbf{d}_i)$, we must have that $Z_i$ is an irreducible component of $\text{rep}_{Q}(I', \textbf{d}_i)$
as $\dim \text{GL}(\textbf{d}_i)$ is the maximal dimension for an irreducible component by Lemma
\ref{lem:dim}.

Viewing $Z$ as in (\ref{eqn:irrdirsum}), it remains to show $\dim Z=\dim \text{GL}(\textbf{d})$. By hypothesis, a general element $M\in \text{rep}_Q(I, \textbf{d})$ is a direct sum of Hom-orthogonal Schur representations; therefore, 
\begin{equation}
\dim \text{End}(M)=m=\dim \text{Stab}_{\text{GL}(\textbf{d})}(M).
\end{equation}
The orbit-stabilizer theorem gives 
\begin{equation}
    \dim \text{GL}(\textbf{d})=\dim \text{GL}(\textbf{d})\cdot M+\dim \text{Stab}_{\text{GL}(\textbf{d})}(M)=\dim \text{GL}(\textbf{d})\cdot M+m.
\end{equation}

Further, as $Z$ has a dense $m$-parameter family of distinct orbits, for a general $M\in Z$, we have 
\begin{equation}
    \dim Z=\dim\text{GL}(\textbf{d})\cdot M+m.
\end{equation}
It follows $\dim Z=\dim\text{GL}(\textbf{d})$ and so $Z$ is an irreducible component of $\text{rep}_{Q}(I', \textbf{d})$ by Lemma \ref{lem:dim}.
\end{proof}

\begin{remark}\label{remark:infinite}
    In fact, the assumption that $A$ and $B$ are finite dimensional in the previous lemma is unnecessary as long as we assume each irreducible component of $\text{rep}_Q(I', \textbf{d})$ has dimension at most $\dim \text{GL}(\textbf{d})$, which is the context of Corollory \ref{cor:clannishnormal}.
\end{remark}

We are interested in being able to view irreducible components of clannish algebras as irreducible components of skewed-gentle algebras as we are able to more easily prove irreducible components of skewed-gentle algebras are normal.
Normality is the key geometric property in understanding moduli of
representations of algebras, as evident by the below proposition, as well as Theorem \ref{Thm:DecModSp}(c).

\begin{proposition}\cite[Proposition 7]{gentle}
Let $A=kQ/I$ be a tame finite dimensional algebra and $Z\subset\text{rep}_Q(I, \textbf{d})$ an indecomposable irreducible component. Then the following hold.
\begin{itemize}
    \item For any weight $\theta\in\mathbb{Z}^{Q_0}$ with $Z^{ss}_{\theta}\neq \emptyset$, the variety $\mathcal{M}(Z)^{ss}_{\theta}$ is either a point or a projective curve.
    \item If $\theta\in\mathbb{Z}^{Q_0}$ is so that $Z^s_{\theta}\neq \emptyset$ then $\mathcal{M}(Z)^{ss}_{\theta}$ is rational. If, in addition, $Z$ is normal then $\mathcal{M}(Z)^{ss}_{\theta}$ is either a point or $\mathbb{P}^1$.
\end{itemize}
\end{proposition}

In combination with Theorem \ref{Thm:DecModSp}, the above proposition allows us to conclude $\mathcal{M}(Z)^{ss}_{\theta}$ is a product of projective spaces if we can prove 
$\widetilde{Z}=\overline{Z_1^{\oplus m_1}\oplus\ldots\oplus Z_r^{\oplus m_r}}$ is normal, where $m_1Z_1\dot{+}\ldots\dot{+}m_rZ_r$ is the $\theta$-stable composition of $Z$.

In order to prove the normality of irreducible components of skewed-gentle algebras, and eventually
the normality of irreducible components of clannish algebras, we will need to review gentle algebras
and their irreducible components.

\subsection{Gentle and Special Biserial Algebras}
Known affirmative examples of Conjecture \ref{conj:conjecture} include the classes of gentle algebras and special biserial algebras. While the focus of this paper is proving the conjecture for clannish algebras, the proof of Theorem \ref{thm:projmod} will generalize the main argument made in \cite{ModSpBi}, so it will be helpful to review the basic ideas utilized in that paper. We first define gentle algebras, which will be used later to define a more general class of algebras known as skewed-gentle algebras.

\begin{definition}\label{def:GentlePair}\cite{GePe}
A gentle pair is a pair $(Q, I)$ given by a quiver $Q$ and an ideal $I$ generated by paths of length two in $Q$ such that 
\begin{itemize}
    \item for each $i\in Q_0$, there are at most two arrows with source $i$, and at most two arrows with target $i$;
    \item for each arrow $\alpha:i \rightarrow j$ in $Q_1$, there exists at most one arrow $\beta$ with target $i$ such that $\beta\alpha\in I$ and at most one arrow $\beta'$ with target $i$ such that $\beta'\alpha\not\in I$;
    \item for each arrow $\alpha:i\rightarrow j$ in $Q_1$, there exists at most one arrow $\beta$ with source $j$ such that $\alpha\beta\in I$ and at most one arrow $\beta'$ with source $j$ such that $\alpha\beta'\not\in I$.
    \item the algebra $A=kQ/ I$ is finite dimensional. 
\end{itemize}
An algebra $A=kQ/I$ is said to be gentle if $(Q, I)$ is a gentle pair.
\end{definition}
For $A=kQ/I$ gentle, by \cite[\S 3]{ModSpBi}, the irreducible components of $\text{rep}_Q(I, \textbf{d})$ are all of the form 
\begin{equation}\label{eqn:IrrGentle}
    \text{rep}_Q(I, \textbf{d}, \textbf{r})=\{M\in \text{rep}_{Q}(I, \textbf{d})\,|\, \text{rank } M(a)\leq r_a, \forall a\in Q^*_1\}
\end{equation}
where $\textbf{r}=(r_a)_{a\in Q_1}$ is a maximal rank sequence, with respect to the coordinate wise order, for $\textbf{d}$. It is known that every irreducible component of the above form is locally isomorphic to an affine Schubert variety of type $A$ by \cite[Theorem 11.3]{Lusztig}, which are normal by \cite[Theorem 8]{Faltings}. 
\begin{definition}\cite{SKOW}\label{def:SpBi}
An algebra $A$ is called special biserial if it is Morita equivalent to a bound quiver algebra $kQ/\langle \mathcal{R}\rangle$, where $(Q, \mathcal{R})$ satisfies the following conditions:
\begin{enumerate}
    \item[(SB1)] For each vertex $v\in Q_0$, there are at most two arrows with head $v$ and at most two arrows with tail $v$. 
    \item[(SB2)] For any arrow $b\in Q_1$ there is at most one arrow $a\in Q_1$ with $ba\not\in \mathcal{R}$ and at most one arrow $c\in Q_1$ with $cb\not\in \mathcal{R}$.
\end{enumerate}
\end{definition}

Any special biserial algebra is tame, with the indecomposable representations being either string or band representations \cite{RingelDi, RingelButler}.
A lemma by Ringel, see \cite[\S 4]{ModSpBi} for a proof, shows any special biserial algebra can be viewed as a quotient of a gentle algebra. 

With this in mind, we have a closed embedding $\text{rep}_Q(I, \textbf{d})\xhookrightarrow{} \text{rep}_{Q'}(I', \textbf{d})$ with $A=kQ/I$ special biserial and $kQ'/I'$ gentle. Furthermore, \cite[\S 4]{ModSpBi} shows this closed embedding preserves certain irreducible components through a dimension argument. The previous observation, along with Theorem \ref{Thm:DecModSp} and the paragraph proceeding Definition \ref{def:SpBi} gives the following classification.
\begin{theorem}\cite{ModSpBi}\label{thm:SpBiMod}
Let $A$ be a special biserial algebra. Then any irreducible component of a moduli space $\mathcal{M}(A, \textbf{d})^{ss}_{\theta}$ is isomorphic to a product of projective spaces.
\end{theorem}

In this paper, we generalize the above theorem, as well as the proof idea, to skewed-gentle algebras and clannish algebras. 

\subsection{Clannish Algebras} 
We next introduce two classes of finite dimensional tame algebras, skewed-gentle algebras and clannish algebras, which generalize gentle algebras and special biserial algebras, respectively. 
With $Q$ a quiver, we let $Q_1^{\text{sp}}\subset Q_1$ be a subset of loops of $Q_1$.
We call the elements of $Q_1^{\text{sp}}$ special loops.
Any arrow within $Q_1^{\text{ord}}=Q_1\backslash Q_1^{\text{sp}}$ will be called ordinary. Additionally, if $v\in Q_0$ has a special loop attached, then we call $v$ a special vertex. Otherwise, $v$ will be called ordinary.
When defining a set $\mathcal{R}$ of relations on $Q$, we always include the set of relations 
\begin{equation}\label{eqn:SpecialLoops}
\mathcal{R}^{\text{sp}}=\{e^2-e \, | \, e\in Q_1^{\text{sp}}\},
\end{equation}
so $\mathcal{R}=\mathcal{R}^{\text{sp}}\cup \mathcal{I}$ where $\mathcal{I}$ is a set of zero-relations.

\begin{definition}\label{ref:SkewGentle}\cite{BeHo}
A skewed-gentle triple $(Q, I, Q_1^{\text{Sp}})$ is the data of a quiver $Q$, an ideal $I$ generated by paths of length two in $Q$, and a subset $Q_1^{\text{sp}}$ of special loops in $Q$ such that $(Q, I +\langle e^2, e\in Q_1^{\text{sp}} \rangle)$ is a gentle pair. In this case, the algebra $A=kQ/(I+ \langle\mathcal{R}^{\text{sp}}\rangle) $ is called a skewed-gentle algebra. As a gentle algebra is finite dimensional, so too is a skewed-gentle algebra.
\end{definition}

\begin{definition}{\cite{CB2}}\label{def:Clannish}
With $\mathcal{R}=\mathcal{R}^{\text{sp}}\cup \mathcal{I}$ and $I=\langle \mathcal{R} \rangle$, the algebra $\Lambda = kQ/I$ is clannish when the following conditions hold:

\begin{itemize}
    \item[(C1)] None of the zero-relations in $\mathcal{I}$ begin or end with a special loop, or involve the square of a special loop. 
    \item[(C2)] For each vertex $v\in Q_0$, there are at most two arrows with head $v$ and at most two arrows with tail $v$. 
    \item[(C3)] For any arrow $b\in Q_1\backslash Q_1^{\text{sp}}$ there is at most one arrow $a\in Q_1$ with $ba\not\in \mathcal{I}$ and at most one arrow $c\in Q_1$ with $cb\not\in \mathcal{I}$. The arrows $a, c\in Q_1$ can be either ordinary or special. 
\end{itemize}
\end{definition}

From the definition, one sees clannish algebras are obtained from special biserial algebras by allowing vertices of $Q$ to have an idempotent loop.
A classification of the indecomposable representations over clannish algebras is given in \cite{CB2}. Just like special biserial algebras, the indecomposable representations are either string or band representations, the only difference being that we can now have indecomposable representations that correspond to words that are symmetric strings and symmetric bands. 
Furthermore, it is known that clannish algebras are tame \cite{CB2}.

In general, we do not have that a clannish algebra $\Lambda=kQ/I$ is a bound quiver algebra, for if $e\in Q_1^{\text{sp}}$ then $e^2-e\in \mathcal{R}^{\text{sp}}$ is not admissible;
however, from $Q$ and $\mathcal{R}$, we can construct a quiver $Q'$ and an ideal $I'=\langle \mathcal{R}'\rangle$ such that $\Lambda=kQ/I\cong kQ'/I'=:\Lambda'$ where $\Lambda'$ is a bound quiver algebra. The construction outlined below closely follows \cite{AB19}, and is derived by basic observations regarding the idempotent relations applied to the special loops of clannish algebras.
Namely, given any $e\in Q_1^{\text{sp}}$, we have the isomorphism 
$$k\langle e\rangle/\langle e^2-e\rangle\cong ke\times k(1-e)$$ 
given by the Chinese remainder theorem.
This isomorphism allows one to split any special vertex $v\in Q_0^{\text{sp}}$ into two ordinary vertices $v^-, v^+\in Q'^{\text{ord}}_0$, thus giving a bijection
$$|Q'_0|=|Q_0^{\text{ord}}\cup (Q_0^{\text{sp}}\times \mathbb{Z}_2)|.$$ With regards to moving from $Q$ to $Q'$, there are a few cases to consider: 
\begin{itemize}
    \item Any arrow $a\in Q_1^{\text{ord}}$ incident to ordinary vertices $v, w\in Q^{\text{ord}}_0$ will remain unchanged in $Q'_1$. 
    \item Consider a subquiver of $Q$ as below, where $v\in Q_0^{\text{ord}}$ and $w\in Q_0^{\text{sp}}$
    
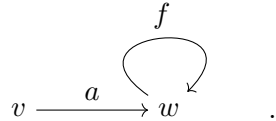
\begin{center}
\[ \begin{tikzpicture}[baseline=-3]
		 \node (1) at (0,0) {$v$};
		 \node (2) at (2,0) {$w$} edge[in=45, out=145, loop] node[above] {$f$} ();
		 \draw[->] (1) -- (2) node[midway, above] {$a$};
		 \end{tikzpicture}. \]
\end{center}
Upon splitting the special vertex $w$, we get two vertices $w^{+}, w^-\in Q'_0$ and two arrows $\prescript{+}{}a, \prescript{-}{} a\in Q'_1$.

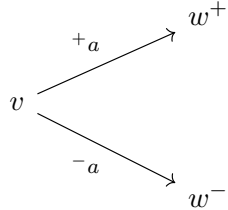
\[\begin{tikzcd}
	&& {w^+} \\
	v \\
	{} && {w^-}
	\arrow["{\prescript{+}{}a}", from=2-1, to=1-3]
	\arrow["{\prescript{-}{}a}"', from=2-1, to=3-3]
\end{tikzcd}\]
    
    \item Similarly, subquivers in $Q$ of the form  
    
    
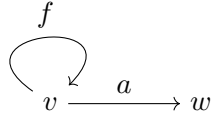
\begin{center}
\[ \begin{tikzpicture}[baseline=-3]
		 \node (1) at (0,0) {$v$} edge[in=45, out=145, loop] node[above] {$f$} ();
		 \node (2) at (2,0) {$w$};
		 \draw[->] (1) -- (2) node[midway, above] {$a$};
		 \end{tikzpicture} \]
\end{center}
with $v\in Q_0^{\text{sp}}$ and $w\in Q_0^{\text{ord}}$ become     
    
    
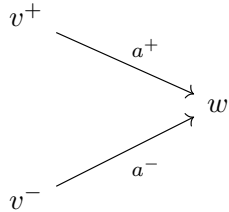
\[\begin{tikzcd}
	{v^+} && {} \\
	&& w \\
	{v^-}
	\arrow["a^-"', from=3-1, to=2-3]
	\arrow["a^+", from=1-1, to=2-3]
\end{tikzcd}\]
upon splitting the special vertex.     
    \item Lastly, subquivers of the form 
    
    
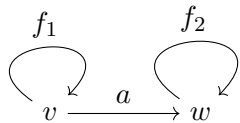
\begin{center}
\[ \begin{tikzpicture}[baseline=-3]
		 \node (1) at (0,0) {$v$} edge[in=45, out=145, loop] node[above] {$f_1$} ();
		 \node (2) at (2,0) {$w$} edge[in=45, out=145, loop] node[above] {$f_2$} ();
		 \draw[->] (1) -- (2) node[midway, above] {$a$};
		 \end{tikzpicture} \]
    \end{center}
with $v,w \in Q_0^{\text{sp}}$ take the form

\[\begin{tikzcd}
	{v^+} && {w^+} \\
	\\
	{v^-} && {w^-}
	\arrow["{\prescript{+}{}a^+}", from=1-1, to=1-3]
	\arrow["{\prescript{-}{}a^-}"', from=3-1, to=3-3]
	\arrow["{\prescript{-}{}a^+}"'{pos=0.2}, from=3-1, to=1-3]
	\arrow["{\prescript{+}{}a^-}"{pos=0.2}, from=1-1, to=3-3]
\end{tikzcd}\]
upon splitting $v$ and $w$ simultaneously. 
\end{itemize}

We again follow \cite{AB19} when constructing the ideal $I'$ from $I$.
\begin{itemize}
    \item When given a relation $ba\in I$ on a subquiver \[\begin{tikzcd}
	\bullet && v && \bullet
	\arrow["a", from=1-1, to=1-3]
	\arrow["b", from=1-3, to=1-5],
\end{tikzcd}\]
if $v$ is an ordinary vertex, then we have $\prescript{\epsilon}{}ba^{\epsilon'}\in I'$ for $\epsilon,\epsilon'\in\{\emptyset, -, +\}$, where the possible superscripts are determined by whether or not the vertices on the left and right are special or ordinary. 
\item If $v$ is a special vertex in the subquiver above, then we have $(\prescript{\epsilon}{}b^{+})(\prescript{+}{}a^{\epsilon'})+(\prescript{\epsilon}{}b^{-})(\prescript{-}{}a^{\epsilon'})\in I'$ for each $\epsilon,\epsilon'\in\{\emptyset, -, +\}$ that makes sense.
\item Given a relation $bfa\in I$ on a subquiver of the form 
    \begin{center}
\[ \begin{tikzpicture}[baseline=-3]
         \node (1) at (0,0) {$w_1$};
		 \node (2) at (2,0) {$v$} edge[in=45, out=145, loop] node[above] {$f$} ();
		 \node (3) at (4,0) {$w_2$};
		 \draw[->] (1) -- (2) node[midway, above] {$a$};
		 \draw[->] (2) -- (3) node[midway, above] {$b$};
		 \end{tikzpicture} \]
\end{center}
with $w_1,w_2\in Q_0^{\text{ord}}$, we have $(b^{+})(\prescript{+}{}a)\in I'$ with $(b^{+})(\prescript{+}{}a)=(b^{-})(\prescript{-}{}a)$ from the above bullet point.
\item The other possible relations of $I'$ can be derived using combinations of the three cases demonstrated above.
\end{itemize}

\begin{remark}
Keeping the same notation as above, we will say $\Lambda=kQ/I$ is the standard presentation of the clannish algebra while $\Lambda'=kQ'/I'$ is the admissible presentation of the clannish algebra. 
\end{remark}

\begin{remark}\label{remark:presentations}
Despite $\Lambda$ and $\Lambda'$ being isomorphic as algebras, their corresponding representation varieties are generally not isomorphic, since these depend on the specific presentation. For example, if one defines

\begin{center}
\[Q = \begin{tikzpicture}[baseline=-3]
		 \node (1) at (0,0) {$v$} edge[in=45, out=145, loop] node[above] {$f$} ();
    \end{tikzpicture}\] 
\end{center}
with $I=\langle f^2-f \rangle$ and $\textbf{d}=(2)$ then 
$$\text{rep}_Q(I, \textbf{d})=\{M\in \text{Mat}_{2\times 2}\,|\, M^2=M\}.$$
Meanwhile, the split vertex quiver $Q'$ consists of two vertices without any arrows, which yields a representation variety consisting of just one point. In the appendix, we briefly explore how the geometry of the two presentations interact.
\end{remark}

The following example demonstrates that we can't always view a clannish algebra as a quotient of a gentle algebra like we can for special biserial algebras.

\begin{example}

Consider $\Lambda=KQ/I$ where $Q$ is the quiver given by 

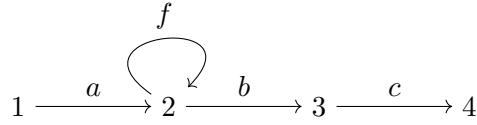
\begin{center}
\[ \begin{tikzpicture}[baseline=-3]
		 \node (1) at (0,0) {1};
		 \node (2) at (2,0) {2} edge[in=45, out=145, loop] node[above] {$f$} ();
		 \node (3) at (4,0) {3};
		 \node (4) at (6,0) {4};
		 \path[->] (1) edge node[above] {$a$} (2)  (2) edge node[above] {$b$} (3)  (3) edge node[above] {$c$} (4);
		 \end{tikzpicture} \]
\end{center}

\noindent and $I=\langle ba, cbfa, f^2-f \rangle$. The algebra $\Lambda$ is isomorphic to the algebra $\Lambda'=kQ'/I'$ given by 

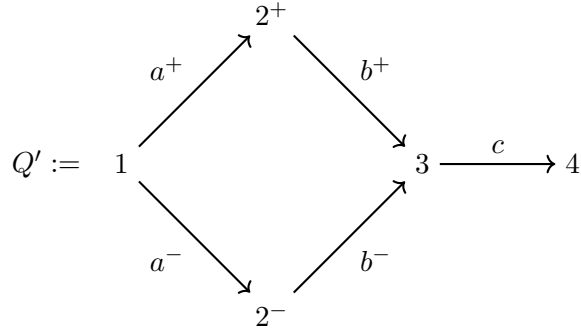
\begin{center}
\begin{tikzpicture}
\node (a) at (0,0) {$1$}
node (b) at (2,-2) {$2^-$}
node (c) at (2,2) {$2^+$}
node (d) at (4,0) {$3$}
node (f) at (-1,0) {$Q':=$}
node (e) at (6,0) {$4$};
\draw[thick, ->] (a) -- (b) node[midway, below left] {$a^-$};
\draw[thick, ->] (a) -- (c) node[midway, above left] {$a^+$};
\draw[thick, ->] (b) -- (d) node[midway, below right] {$b^-$};
\draw[thick, ->] (c) -- (d) node[midway, above right] {$b^+$};
\draw[thick, ->] (d) -- (e) node[midway, above] {$c$};
\end{tikzpicture}
\end{center}
and relations $I'=\langle b^+a^{+}+b^-a^-, cb^+a^+ \rangle$. One sees both $cb^{+}$ and $cb^{-}$ are nontrivial in this setting, which is not a possibility for an algebra that is a quotient of a gentle algebra.

\end{example}

\section{Proof of Main Theorem}
Let $\Lambda=kQ/I$ be a clannish algebra. We would like to understand the geometry of clannish algebras through skewed-gentle algebras; however, not every clannish algebra is a quotient of a skewed-gentle algebra, so instead, we construct an infinite dimensional complete skewed-gentle cover, which will serve a similar purpose to the complete gentle algebras used in \cite{ModSpBi}.

\begin{definition}\label{def:completeskew}
    Let $Q^*$ be a quiver with a distinguished set $Q_1^{*,\text{sp}}$ of special loops. Let $R^*$ be an ideal of relations defined by $$R^*=R^*_{\text{ord}}+\{e^2-e\,|\, e\in Q_1^{*,\text{sp}}\}$$ where $R_{\text{ord}}^*$ is a finite set of monomial length two zero-relations including only ordinary arrows.

    We call $(Q^*, R^*, Q_1^{*,\text{sp}})$ complete skewed-gentle if:
    \begin{enumerate}
        \item Every vertex has exactly two incoming arrows and two outgoing arrows, where a special loop counts as one incoming and one outgoing arrow.
        \item Every ordinary arrow belongs to a unique effective oriented cycle of ordinary arrows in the sense used for complete gentle algebras. 
        \item For every ordinary arrow $a$ there exists exactly one ordinary arrow $b$ with $ab\in R^*_{\text{ord}}$ and exactly one ordinary arrow $c$ with $ca\in R^*_{\text{ord}}$.
        \item Special loops satisfy only $e^2=e$ and no zero-relation starts or ends with a special loop. 
    \end{enumerate}
\end{definition}

\begin{lemma}\label{lem:clantoskewgentle}
Let $\Lambda=kQ/I$ be a clannish algebra in its standard presentation. Then there exists a complete skewed-gentle algebra $\Gamma=kQ^*/\langle \mathcal{R}^* \rangle$ with $Q_0^*=Q_0$ containing $Q$ as a subquiver, such that $\Lambda$ is a quotient of $\Gamma$.
\end{lemma}

\begin{proof}
Let $\Lambda=kQ/I$ be clannish. We first construct $Q^*$ and $\mathcal{R}^*$ and then we define the quotient map $\overline\varphi:kQ^*/\langle \mathcal{R}^* \rangle \rightarrow \Lambda$. We keep the same vertex set $Q_0^*=Q_0$ and all of the special loops $Q_1^{*,\text{sp}}=Q_1^{\text{sp}}$. We then add ordinary arrows until the following conditions hold: 
\begin{itemize}
    \item If $v$ is ordinary, then exactly two arrows end at $v$ and exactly two arrows start at $v$.
    \item If $v$ is special, then exactly one ordinary arrow ends at $v$ and exactly one ordinary arrow starts at $v$.
\end{itemize}

When defining zero-relations at a special vertex $s$ with special loop $e$ in $Q^*$, there exists exactly one ordinary arrow $\alpha$ ending at $s$ and exactly one arrow $\beta$ starting at $s$. As relations cannot begin or end with a special loop, we must have $\beta\alpha\in\mathcal{R}^*$. We do not impose any zero-relation on $e$, only the idempotent relation $e^2-e\in \mathcal{R}^*$.

At ordinary vertices, we use the usual completion procedure from the special biserial case. Namely, after adding ordinary arrows so that exactly two ordinary arrows enter and exactly two ordinary arrows leave the vertex, the clannish condition gives a partial matching of incoming arrows with outgoing arrows corresponding to the nonzero original length-two paths. We extend this partial matching to a bijection, and impose zero-relations on the complementary length-two paths \cite{RingelStrings}.

On an ordinary vertex $v$, there are exactly two arrows $\alpha_1, \alpha_2$ ending at $v$ and exactly two arrows 
$\beta_1, \beta_2$ starting at $v$. For every ordinary vertex $v$, we define $$M_v\subseteq\{\beta_j\alpha_i\, |\, i,j\in \{1,2\}\}$$ such that if $\beta_j\alpha_i\in M_v$ and both $\alpha_i, \beta_j$ are arrows in $Q_1$, then $\beta_j\alpha_i\in I$. We then put each path in $M_v$ into $\mathcal{R}^*$. Doing this independently at every ordinary vertex defines the ordinary length-two zero-relations of $\mathcal{R}^*$.

We define an algebra homomorphism $\varphi:kQ*\rightarrow \Lambda$ on arrows by \[\varphi(a)=\begin{cases}
    a \quad a\in Q_1\\
    0 \quad a\in Q_1^*\setminus Q_1
\end{cases}
\] and on vertices $\varphi$ is the identity. 

Under $\varphi$, every relation of $\mathcal{R}^*$ maps to zero in $\Lambda$. If $e$ is a special loop, then $\varphi(e^2-e)=e^2-e=0$ in $\Lambda$, by definition of a clannish algebra. Second, if $\beta\alpha\in\mathcal{R}^*$ is a length-two zero-relation, then if either $\alpha$ or $\beta$ is a newly added arrow, then $\varphi(\beta\alpha)=0.$ If both $\alpha$ and $\beta$ are original arrows, then by the construction of $\mathcal{R}^*$, $\beta\alpha\in I$, so again $\varphi(\beta\alpha)=0$ in $\Lambda$. It follows that $\langle\mathcal{R}^*\rangle\subseteq \ker \varphi$. Hence $\varphi$ induces a surjective algebra homomorphism $$\overline{\varphi}:\Gamma=kQ^*/\langle \mathcal{R}^* \rangle\rightarrow \Lambda.$$ Surjectivity is immediate as every vertex and every arrow of $Q$ lies in the image of $\varphi$. Thus $\Lambda$ is a quotient of the complete skewed-gentle algebra $\Gamma$.
\end{proof}

Due to the idempotent relations imposed on the quiver by special loops, in order to prove irreducible components of skewed-gentle algebras are normal, we will first show varieties of idempotent matrices are smooth.  

Establishing notation, we define 
\begin{equation}\label{eqn:idempotentdef}
E_n=\{M\in \text{Mat}_{n\times n}(k)\,|\, M^2=M\}
\end{equation}
to be the affine variety of $n\times n$ idempotent matrices.
One has that every $M\in E_n$ is conjugate to some $A_r, 0\leq r\leq n$, where \[A_r=
\left(\begin{array}{@{}c|c@{}}
    I_r & \bigzero \\ \hline
    \bigzero & \bigzero \\ 
  \end{array}\right).
\] 
For a fixed rank $r$, we write $$E_n(r)=GL_n(k)\cdot A_r.$$ In other words, $E_n$ is a priori the disjoint union of sets 
\begin{equation}\label{eqn:idempotents}
    E_n=\bigsqcup_{r=0}^n E_n(r)=\bigsqcup_{r=0}^n \text{GL}_n(k)\cdot A_r.
\end{equation}
Below we show the above is in fact a disjoint union of the connected components of the variety. 

\begin{lemma}\label{lem:idempotentsmooth}
The variety of idempotent matrices, $E_n$, is smooth.
\end{lemma}
\begin{proof}
Using the notation established above, we observe each $\text{GL}_n(k)\cdot A_r$ is an orbit, thus smooth. It remains showing each $\text{GL}_n(k)\cdot A_r$ is a distinct connected component in order to conclude $E_n$ is smooth. 

To prove each $\text{GL}_n(k)\cdot A_r$ is a connected component, it suffices to show they are each closed in the Zariski topology. This follows by realizing 
\begin{equation}
    \text{GL}_n(k)\cdot A_r=\{M\in \text{Mat}_{n\times n}(k)\,|\, M^2=M,\ \text{Tr}(M)=r\}.
\end{equation}
As an intersection of closed sets, we conclude $\text{GL}_n(k)\cdot A_r$ is closed for each $0\leq r\leq n$, and the lemma follows.
\end{proof}

\begin{lemma}\label{lem:idempotentdim}
    For a fixed rank $r$, $$\dim E_n(r)=2r(n-r)\leq \frac{n^2}{2}.$$
\end{lemma}
\begin{proof}
    Write $k^n\cong U\oplus W$ where $\dim U=r,$ $\dim W=n-r$ and $A_r$ is the projection onto $U$ along $W$. An element $g\in \text{GL}_n$ stabilizes $A_r$ under conjugation if and only if $gA_rg^{-1}=A_r$, or rather $gA_r=A_rg$. In block form, we may write $$g=\begin{pmatrix}
        P && Q\\
        R && S
    \end{pmatrix}$$ thus $$gA_r=\begin{pmatrix}
        P && 0\\
        R && 0
    \end{pmatrix}\qquad\text{and} \qquad A_rg=\begin{pmatrix}
        P && Q\\
        0 && 0
    \end{pmatrix}$$ which yields $$\text{Stab}_{\text{GL}_n}(A_r)=\left\{\begin{pmatrix}
        P && 0\\
        0 && S
    \end{pmatrix}\, |\, P\in \text{GL}_r, S\in \text{GL}_{n-r}\right\}.$$ By orbit-stabilizer, we have \begin{align*}
        \dim E_n(r)&=\dim \text{GL}_n-\dim (\text{GL}_r\times \text{GL}_{n-r})\\
        &= n^2 - (r^2+(n-r)^2)\\
        &= 2r(n-r)\\
        &= \frac{n^2}{2}-2(r-\frac{n}{2})^2\\
        &\leq \frac{n^2}{2}
    \end{align*}
\end{proof}

\begin{proposition}\label{lem:skewgentlenormal}
Let $\Gamma=kQ^*/\langle\mathcal{R}^*\rangle$ be a complete skewed-gentle algebra and $\textbf{d}$ be a dimension vector. If $Z\subset \text{rep}_{Q^*}(\mathcal{R^*},\textbf{d})$ is an irreducible component, then $Z$ is normal.
\end{proposition}
\begin{proof}
Let $Z\subset\text{rep}_{Q^*}(\mathcal{R^*}, \textbf{d})$ be an irreducible component
and $\Gamma_1=kQ^*_1/\mathcal{R}_1^*, \ldots, \Gamma_t=kQ^*_t/\mathcal{R}_t^*$
be the $\rho$-blocks of $\Lambda'$.
The isomorphism of affine varieties given by Equation (\ref{eqn: rho-block})
gives a bijection on the set of irreducible components of $\text{rep}_{Q^*}(\mathcal{R}^*, \textbf{d})$
with the set of irreducible components of $\displaystyle\text{rep}_{Q_1^*}(\mathcal{R}_1^*, \pi_1(\textbf{d}))\,\times\,\ldots\,\times\, \text{rep}_{Q_t^*}(\mathcal{R}_t^*, \pi_t(\textbf{d}))$
via $Z\mapsto (\pi_1(Z),\ldots,\pi_t(Z))$.

As $\Gamma$ is complete skewed-gentle, each relation is either a path of length two or an idempotent relation enforced on a special loop.
Additionally, no relation can begin or end with a special loop by condition (C1) of
Definition \ref{def:Clannish}; therefore, the equivalence classes involving special
loops contain only a single special loop.

Letting $e_1, e_2,\ldots,e_l$ denote the
special loops of $Q$, by reordering if necessary, we can let
$\Gamma_1,\ldots,\Gamma_l$ be the corresponding $\rho$-blocks, where $\Gamma_i$
contains only the equivalence class $[e_i]$ for $1\leq i\leq l$ with $\mathcal{R}^*_i=\langle
e_i^2-e_i \rangle$. For each $1\leq i\leq l$, we have $\pi_i(Z)\subset \text{rep}_{Q^*_i}(\mathcal{R}^*_i, \pi_i(\textbf{d}))$ is a variety of idempotent matrices, so using the notation from (\ref{eqn:idempotentdef}), we may write $\pi_i(Z)\cong E_{\pi_i(\textbf{d})}(\textbf{r}_i)$ for some rank $0\leq \textbf{r}_i \leq \pi_i (\textbf{d})$.
By Lemma \ref{lem:idempotentsmooth},
 each $E_{\pi_i(\textbf{d})}(\textbf{r}_i)$ is smooth, so the product $\prod_{i=1}^l \pi_i(Z)\cong \prod_{i=1}^l E_{\pi_i(\textbf{d})}(\textbf{r}_i)$ is smooth as well.

With the $\rho$-blocks corresponding to special loops out of the way,
the remaining $\rho$-blocks $\Gamma_{l+1},\ldots, \Gamma_{t}$
are all gentle algebras. It follows from (\ref{eqn:IrrGentle})
that for each $l+1\leq k\leq t$, the irreducible component $\pi_k(Z)\subset \text{rep}_{Q^*_k}(\mathcal{R}^*_k, \pi_k(\textbf{d}))$
can be expressed as $\text{rep}_{Q^*_k}(\mathcal{R}^*_k, \pi_k(\textbf{d}), \textbf{r}_k)$
where $\textbf{r}_k$ is a maximal rank sequence, with respect to the coordinate wise order, for $\pi_k(\textbf{d})$.
From the discussion around (\ref{eqn:IrrGentle}), such irreducible components are known to be normal, so the $\pi_k(Z)$ are normal for all $l+1\leq k\leq t$.

Writing $Z$ as the product
\begin{equation}
Z\cong\prod_{i=1}^t \pi_i(Z)\cong \prod_{i=1}^l E_{\pi_i(\textbf{d})}(\textbf{r}_i)\times \prod_{k=l+1}^t \text{rep}_{Q^*_k}(\mathcal{R}^*_k, \pi_k(\textbf{d}), \textbf{r}_k)
\end{equation}
one finds that $Z$ is a product of normal varieties and is thus normal.
\end{proof}

\begin{proposition}\label{prop:dimbound}
    Let $\Gamma$ be complete skewed-gentle and let $Z$ be an irreducible component of $\text{rep}_{Q^*}(\mathcal{R}^*, \textbf{d})$, then $$\dim Z\leq \dim \text{GL}(\textbf{d}).$$
\end{proposition}
\begin{proof}
    As proven in Proposition \ref{lem:skewgentlenormal}, we can view an irreducible component $Z$ as a product $$Z\cong \prod_{i=1}^l E_{\pi_i(\textbf{d})}(\textbf{r}_i)\times \prod_{k=l+1}^t \text{rep}_{Q^*_k}(\mathcal{R}^*_k, \pi_k(\textbf{d}), \textbf{r}_k).$$ In \cite{ModSpBi}, it is shown $$\dim\text{rep}_Q(\mathcal{R}, \textbf{d}, \textbf{r})\leq \sum_{i\in Q_0}\textbf{d}(i)^2=\dim\text{GL}(\textbf{d}).$$ For an ordinary circular-complex factor, we use the same bound over all vertices $v$ in the block: $$\dim\text{Comp}(\textbf{n}, \textbf{r})\leq \frac{1}{2}\sum_{v}\textbf{d}(v)^2,$$ and for a special idempotent factor, $E_n(r)=\text{GL}_n\cdot A_r$, at a special vertex $s$, we have $$\dim E_n(r)=2r(n-r)\leq \frac{\textbf{d}(s)^2}{2}$$ by Lemma \ref{lem:idempotentdim}. 

    Since a special loop occupies one of the incoming and outcoming slots, the ordinary effective cycles plus the idempotent factor still count each vertex exactly twice. Adding the two bounds, we conclude $$\dim Z\leq \dim\text{GL}(\textbf{d}).$$
\end{proof}

\begin{corollary}\label{cor:clannishnormal}
Let $\Lambda=kQ/I$ and $\Gamma=kQ^*/\langle \mathcal{R}^* \rangle$ be as above. Let $Z_i\subseteq\text{rep}_Q(I, \textbf{d}_i),
1\leq i\leq m$, be irreducible components 
satisfying:
\begin{itemize}
    \item each $Z_i$ is a Schur component;
    \item each $c_A(Z_i) = 1$;
    \item $\Hom_A(M_i, M_j) = 0$ for $i \neq j$ and general points $M_i \in Z_i$, $M_j \in Z_j$.
\end{itemize}
Letting $\textbf{d}=\sum_{i=1}^m\textbf{d}_i$, then $Z=\overline{Z_1\oplus\ldots\oplus Z_m}$ is an irreducible
component of $\text{rep}_{Q^*}(\mathcal{R}^*, \textbf{d})$ with regard to the
closed embedding $\text{rep}_Q(I, \textbf{d})\subset
\text{rep}_{Q^*}(\mathcal{R}^*, \textbf{d})$. As such, $Z$ is normal.
\end{corollary}
\begin{proof}
After setting the added arrows when extending $Q$ to $Q^*$ to zero, we may view $Z$ as an irreducible component of $\text{rep}_{Q^*}(\mathcal{R}^*, \textbf{d})$ by combining Lemma \ref{lem:clantoskewgentle} and Proposition \ref{prop:dimbound} with Lemma \ref{lem:tameirred} and Remark \ref{remark:infinite}. Proposition \ref{lem:skewgentlenormal} then allows us to conclude $Z$ is normal.
\end{proof}

\begin{proof}[Proof of Theorem \ref{thm:projmod}]
Let $Y$ be an irreducible component of the projective variety $\mathcal{M}(\Lambda, \textbf{d})^{ss}_{\theta}$. Then there exists an irreducible component $Z\subset \text{rep}_Q(I, \textbf{d})$ such that $Y=\mathcal{M}(Z)^{ss}_{\theta}$. Let  $Z$ have $\theta$-stable decomposition
\begin{equation}
    Z=m_1Z_1\dot{+}\cdots\dot{+}m_rZ_r.
\end{equation}

By parts (a) and (b) of Theorem \ref{Thm:DecModSp}, we may assume none of the $Z_i$ are orbit closures
and $Z=\overline{Z_1^{\oplus m_1}\oplus\cdots\oplus Z_r^{m_r}}$. We prove $Z$ is normal via Corollary \ref{cor:clannishnormal}.

We first note that as each $Z_i$ is $\theta$-stable, a general element of $Z_i$ is Schur. Additionally, as none of the $Z_i$ are orbit closures, we have $\text{hom}(Z_i, Z_j)=0$ for all $1\leq i, j \leq m$.
Indeed, as each $Z_i$ is $\theta$-stable, one has $\text{hom}(Z_i, Z_j)=0$ whenever
$i\neq j$, and when $i=j$, just choose two non-isomorphic points $M_{\lambda}, M_{\mu}$
with $\lambda\neq \mu$ in the rational curve $C$ from Theorem \ref{thm:tameirreducible}.
As $\text{Hom}(M_{\lambda}, M_{\mu})=0$, it follows $\text{hom}(Z_i, Z_i)=0$. 

We have proven that the general element of $Z$ is a direct sum of Hom-orthogonal Schur band components, so Corollary \ref{cor:clannishnormal} gives that $Z$ is normal. As normality is preserved by quotients, $\mathcal{M}(Z)^{ss}_{\theta}$ is normal as well. 
By Theorem \ref{Thm:DecModSp}(c), we get the desired decomposition
\begin{equation*}
    \mathcal{M}(Z)^{ss}_{\theta}\cong \prod_{i=1}^r S^{m_i}(\mathcal{M}(Z_i)^{ss}_{\theta})\cong \prod_{i=1}^r\mathbb{P}^{m_i}. \qedhere
\end{equation*}
\end{proof}

\section{Future Directions}
\begin{enumerate}
\item Is there a class of algebras akin to clannish algebras for quivers of type $\widetilde{\mathbb{E}}$?  If so, would we be able to generalize the arguments of this paper to conclude that their moduli spaces are products of projective spaces?
\item Rather than focusing on a specific class of tame algebra, one can instead focus on a particularly nice irreducible component for arbitrary tame algebras. With $kQ/I$ tame acyclic and  $$\mathcal{P}_A(\textbf{d}):=\{M\in \text{rep}_Q(I, \textbf{d})\,|\, \text{pdim}_A M\leq 1\}$$ one has that $\mathcal{C}_A(\textbf{d})=\overline{\mathcal{P}_A(\textbf{d})}$ is an irreducible component. The following problem was posed by Calin Chindris.

\begin{problem}
Let $A$ be acyclic and tame. 
If $\textbf{d}\in\mathbb{N}^{Q_0}$ is such that $\mathcal{P}_A(\textbf{d})\neq \emptyset$, describe $\mathcal{M}(\mathcal{C}_A(\textbf{d}))^{\text{ss}}_{\langle\textbf{d},- \rangle}$.\\

\end{problem}

\end{enumerate}

\section{Appendix: The Geometry of  Standard and Admissible Presentations}

When discussing the standard and admissible presentations of a clannish algebra, we demonstrated that the representation varieties of the two presentations are generally not isomorphic to one another (see Remark \ref{remark:presentations}). In this appendix, we explore how the irreducible components of the representation varieties of the two presentations relate to one another. In particular, we find there is a correspondence between irreducible components of the two presentations and these corresponding components are normal simultaneously. Ultimately, regardless of the presentation, they descend to the same moduli problem. The techniques and machinery used in this section bears some resemblance to the work involving node-splitting found in Sections 3 and 4 of \cite{KinAnd}, despite special vertices not being nodes. 

To begin, let $\Lambda=kQ/I$ be a clannish algebra in standard presentation with special loops $e_v:v\rightarrow v$ at special vertices $v\in Q_0^{\text{sp}}$ and let $\Lambda'=kQ'/I'$ be the admissible presentation with split special vertices. Fix a dimension vector $\textbf{d}\in \mathbb{N}^{Q_0}$. For each special vertex, a representation $M$ assigns to $e_v$ an idempotent matrix $E_v=M(e_v)\in \text{Mat}_{d(v)\times d(v)}$.

Fix a rank vector $\textbf{r}=(r_v)_{v\in Q_0^{\text{sp}}}$, $0\leq r(v)\leq d(v)$. Let $\text{rep}_Q^{\textbf{r}}(I, \textbf{d})$ be the locus where $\text{rank}(E_v)=r_v$ for every special vertex $v$. We may write $k^{d_v}=\text{im}E_v\oplus \text{ker}E_v$ at each special vertex $v$. The rank vector $\textbf{r}$ then determines a dimension vector $\textbf{d}^{\textbf{r}}$ for the split quiver $Q'$ by $$\textbf{d}^{\textbf{r}}(v^+)=r_v, \qquad \textbf{d}^{\textbf{r}}(v^-)=d(v)-r_v$$ on the split special vertices and $\textbf{d}^{\textbf{r}}(w)=\textbf{d}(w)$ on the ordinary vertices.

After applying a change of basis at each special vertex, every representation can be moved to one where $M(e_v)=P_v$ for all $v\in Q_0^{\text{sp}}$, where $P_v=\begin{pmatrix} I_{r_v} && 0\\
0 && 0\end{pmatrix}$. Moreover, in Section 4, we observed that we have $E_{d(v)}(r_v)=\text{GL}_{d(v)}\cdot P_v$ with $\dim E_n(r)=2r(n-r)$. In particular, $$E_n(r)\cong \text{GL}_n/(\text{GL}_{r}\times \text{GL}_{n-r}).$$

Let $S_{\textbf{r}}\subseteq \text{rep}^{\textbf{r}}_Q(I,\textbf{d})$ be the closed slice consisting of representations satisfying $M(e_v)=P_v$ for every special vertex $v$ and decompose each vector space at a special vertex as $$k^{d(v)}\cong k^{r_v}\oplus k^{d(v)-r_v}.$$ 

With respect to the arrows, the corresponding linear transformations decompose into blocks. For example, if we have an arrow $a:u\rightarrow v$ with $u$ ordinary and $v$ special, then $M(a):k^{d(u)}\rightarrow k^{r_v}\oplus k^{d(v)-r_v}$ decomposes as $$M(a)=\begin{pmatrix}
    M(a^+)\\
    M(a^-)
\end{pmatrix}$$ where $a^+:u\rightarrow v^+$ and $a^-:u\rightarrow v^-$. Similarly, an arrow between two special vertices decomposes into block maps involving the arrows $\prescript{+}{}a^+, \prescript{+}{}a^-, \prescript{-}{}a^+, \prescript{-}{}a^-$. Naturally, we find $S_\textbf{r}\cong \text{rep}_{Q'}(I', \textbf{d}^{\textbf{r}})$. 

Let $\displaystyle{G_{\textbf{r}}=\prod_{v\in Q_0^{\text{sp}}}} \text{GL}_{d(v)}$ act by changing basis at the special vertices and let \\$\displaystyle{H_{\textbf{r}}=\prod_{v\in Q_0^{\text{sp}}}(\text{GL}_{r_v}\times \text{GL}_{d(v)-r_v})}$ be the stabilizer of the standard idempotents $P_v$, then 
\begin{align*}
    \text{rep}^{\textbf{r}}_{Q}(I, \textbf{d})&\cong G_\textbf{r}\times_{H_{\textbf{r}}} S_{\textbf{r}}\\
    &\cong G_{\textbf{r}}\times_{H_{\textbf{r}}} \text{rep}_{Q'}(I', \textbf{d}^{\textbf{r}})
\end{align*} Thus $\text{rep}^{\textbf{r}}_{Q}(I, \textbf{d})$ is an associated bundle over the smooth homogeneous space $G_{\textbf{r}}/H_{\textbf{r}}$ with fiber $\text{rep}_{Q'}(I', \textbf{d}^{\textbf{r}})$. 

As $G_{\textbf{r}}/H_{\textbf{r}}$ is smooth and irreducible, the irreducible components of $G_{\textbf{r}}\times_{H_{\textbf{r}}} \text{rep}_{Q'}(I', \textbf{d}^{\textbf{r}})$ are obtained from the irreducible components of $\text{rep}_{Q'}(I', \textbf{d}^{\textbf{r}})$, so if $Z'\subseteq \text{rep}_{Q'}(I', \textbf{d}^{\textbf{r}})$ is irreducible, so too is $$Z=G_{\textbf{r}}\times_{H_{\textbf{r}}} Z'.$$ Further, the projection $Z: G_{\textbf{r}}\times_{H_{\textbf{r}}} Z'\rightarrow G_{\textbf{r}}/H_{\textbf{r}}$ is locally trivial in the Zariski topology with fiber $Z'$. As $G_{\textbf{r}}/H_{\textbf{r}}$ is smooth, locally $Z$ looks like $U\times Z'$ for a smooth open set $U$. As normality is stable under product with a smooth variety and can be checked locally, it follows that $Z$ is normal if and only if $Z'$ is normal.

Letting $\theta\in \mathbb{Z}^{Q_0}$ be a weight for the standard presentation, then it induces a weight $\theta^{\textbf{r}}$ on the admissible presentation by setting $\theta^{\textbf{r}}(w)=\theta(w)$ whenever $w$ is an ordinary vertex and $$\theta^{\textbf{r}}(v^+)=\theta(v), \qquad \theta^{\textbf{r}}(v^-)=\theta(v)$$ for each special vertex $v$. With this choice, we have $\theta\cdot \textbf{d}=\theta^{\textbf{r}}\cdot \textbf{d}^{\textbf{r}}$ and subrepresentations correspond under the the idempotent decomposition. It follows that $M\in Z$ is $\theta$-semistable if and only if $M'\in Z'$ is $\theta^{\textbf{r}}$-semistable.

Finally, if $Z$ has $\theta$-stable decomposition $$Z=m_1Z_1\dot{+}\cdots\dot{+}m_rZ_r$$ then $Z'$ has an analogous $\theta^{\textbf{r}}$-stable decomposition $$Z'=m_1Z'_1\dot{+}\cdots\dot{+}m_rZ'_r.$$ Applying the moduli-decomposition theorem to both, we have $$\mathcal{M}(Z)^{ss}_{\theta}\cong \mathcal{M}(Z')^{ss}_{\theta^{\textbf{r}}}\cong\prod_i \mathbb{P}^{m_i}.$$

\bibliographystyle{alpha}
\bibliography{bibliography.bib}

\end{document}

%% file: preamble.tex
\usepackage{enumitem, amsmath, amsthm, amsfonts, amssymb, xy,  mathrsfs, graphicx, lscape}
\usepackage[usenames, dvipsnames]{color}
\usepackage[margin=1in]{geometry} 
\usepackage{tikz}
\usepackage{pgfplots}
\usepackage{mathtools}
\usepackage{comment}
\usetikzlibrary{matrix,arrows, calc}
\pgfplotsset{compat=1.17}
\usepackage{tikz-cd}
\usepackage{xcolor}
\usepackage{stackengine}

\theoremstyle{definition}

\numberwithin{equation}{section}
\newtheorem{theorem}[equation]{Theorem}

\newtheorem{proposition}[equation]{Proposition}
\newtheorem{lemma}[equation]{Lemma}
\newtheorem{corollary}[equation]{Corollary}
\newtheorem{conjecture}[equation]{Conjecture}
\newtheorem{problem}[equation]{Problem}

\newtheorem{rmk}[equation]{Remark}
\newenvironment{remark}[1][]{\begin{rmk}[#1] }{\popQED \end{rmk}}
\newtheorem{eg}[equation]{Example}
\newenvironment{example}[1][]{\begin{eg}[#1] \pushQED{\qed}}{\popQED \end{eg}}
\newtheorem{defn}[equation]{Definition}
\newenvironment{definition}[1][]{\begin{defn}[#1]}{\popQED \end{defn}}
\newtheorem{ques}[equation]{Question}

\usepackage[in]{fullpage}
\usepackage[colorlinks=true, pdfstartview=FitV, linkcolor=black, citecolor=black, urlcolor=black]{hyperref}

\renewcommand{\phi}{\varphi}
\renewcommand{\emptyset}{\varnothing}

\newcommand{\bigzero}{\mbox{\normalfont\Large\bfseries 0}}
\newcommand{\rvline}{\hspace*{-\arraycolsep}\vline\hspace*{-\arraycolsep}}

\DeclareMathOperator{\rep}{rep}
\DeclareMathOperator{\Hom}{Hom}

\newcommand{\cody}[1]{\textcolor{purple}{$[\star$ Cody: #1 $\star]$}}

\makeatletter
\def\Ddots{\mathinner{\mkern1mu\raise\p@
\vbox{\kern7\p@\hbox{.}}\mkern2mu
\raise4\p@\hbox{.}\mkern2mu\raise7\p@\hbox{.}\mkern1mu}}
\makeatother